\documentclass[reqno]{amsart}

\makeatletter
\def\@tocline#1#2#3#4#5#6#7{\relax
  \ifnum #1>\c@tocdepth %
  \else
    \par \addpenalty\@secpenalty\addvspace{#2}%
    \begingroup \hyphenpenalty\@M
    \@ifempty{#4}{%
      \@tempdima\csname r@tocindent\number#1\endcsname\relax
    }{%
      \@tempdima#4\relax
    }%
    \parindent\z@ \leftskip#3\relax \advance\leftskip\@tempdima\relax
    \rightskip\@pnumwidth plus4em \parfillskip-\@pnumwidth
    #5\leavevmode\hskip-\@tempdima
      \ifcase #1
       \or\or \hskip 1em \or \hskip 2em \else \hskip 3em \fi%
      #6\nobreak\relax
    \hfill\hbox to\@pnumwidth{\@tocpagenum{#7}}\par%
    \nobreak
    \endgroup
  \fi}
\makeatother

\usepackage[export]{adjustbox}
\usepackage{thmtools}
\usepackage{amsmath}
\usepackage{amsfonts}
\usepackage{amssymb}
\usepackage{amsthm}
\usepackage{enumitem}
\usepackage{bm}
\usepackage{hyperref}
\usepackage{graphicx,color}
\usepackage{tikz}
\usetikzlibrary{positioning}%
\usetikzlibrary{arrows}%
\usetikzlibrary{calc}
\usepackage{float}
\usepackage{color}
\usepackage{tikz-cd}
\usepackage{cleveref}
\usepackage{mathtools}
\usepackage{rotating}
\usepackage{xcolor,colortbl}
\usepackage[figure]{hypcap}
\usepackage[font=small]{caption}
\usepackage{lipsum}
\usepackage[utf8]{inputenc}
\theoremstyle{plain}
\newtheorem{theorem}{Theorem}[section]
\newtheorem{prop}[theorem]{Proposition}

\newtheorem{corollary}[theorem]{Corollary}

\newtheorem{question}[theorem]{Question}

\theoremstyle{definition}
\newtheorem{definition}[theorem]{Definition}
\declaretheorem[sibling=theorem,name=Example,qed={$\diamondsuit$}]{example}
\declaretheorem[sibling=theorem,name=Remark,qed={$\diamondsuit$}]{remark}

\numberwithin{theorem}{section}
\numberwithin{equation}{section}

\DeclareMathOperator{\rk}{rk}

\def\CC{\mathbb{C}}
\def\QQ{\mathbb{Q}}
\def\RR{\mathbb{R}}
\def\ZZ{\mathbb{Z}}

\def\NN{\mathbb{N}}

\def\Z+{\mathbb{Z}_{\geq 0}}
\def\R+{\mathbb{R}_{\geq 0}}

\makeatletter
\renewcommand*\env@matrix[1][\arraystretch]{%
  \edef\arraystretch{#1}%
  \hskip -\arraycolsep
  \let\@ifnextchar\new@ifnextchar
  \array{*\c@MaxMatrixCols c}}
\makeatother

\newcommand\restr[2]{{%
		\left.\kern-\nulldelimiterspace %
		#1 %
		\vphantom{\big|} %
		\right|_{#2} %
}}

\makeatletter
\@namedef{subjclassname@2020}{%
	\textup{2020} Mathematics Subject Classification}
\makeatother


\title{Quasihomomorphisms from the integers into Hamming metrics}

\author[Draisma]{Jan Draisma}
\address{Mathematical Institute, University of Bern, Sidlerstrasse 5, 3012 Bern, Switzerland; and Department of Mathematics and Computer Science, Eindhoven University of Technology, P.O. Box 513, 5600MB, Eindhoven, the Netherlands}
\email{jan.draisma@unibe.ch}

\author[Eggermont]{Rob H. Eggermont}
\address{Department of Mathematics and Computer Science, Eindhoven University of Technology, P.O. Box 513, 5600MB, Eindhoven, the Netherlands}
\email{r.h.eggermont@tue.nl}

\author[Seynnaeve]{Tim Seynnaeve}
\address{Mathematical Institute, University of Bern, Alpeneggstrasse 22, 3012 Bern, Switzerland}
\email{tim.seynnaeve@unibe.ch}

\author[Tairi]{Nafie Tairi}
\address{Mathematical Institute, University of Bern, Alpeneggstrasse 22, 3012 Bern, Switzerland}
\email{nafie.tairi@unibe.ch}

\author[Ventura]{Emanuele Ventura}
\address{
Politecnico di Torino, 
Dipartimento di Scienze Matematiche ``G.L. Lagrange'', Corso Duca degli Abruzzi 24
10129 Torino, Italy}
\email{emanuele.ventura@polito.it, emanueleventura.sw@gmail.com}

\thanks{JD,TS,NT,EV were partially or fully supported by Vici grant 639.033.514 from the Netherlands Organisation for Scientific Research (NWO) and Swiss National Science Foundation (SNSF) project grant 200021\_191981. RE was supported by NWO Veni grant 016.Veni.192.113.}

\begin{document}
\maketitle

\begin{abstract}
A function $f: \ZZ \to \QQ^n$ is a {\it $c$-quasihomomorphism} if the
Hamming distance between $f(x+y)$ and $f(x)+f(y)$ is at most $c$ for all
$x,y \in \ZZ$.  We show that any $c$-quasihomomorphism has distance at
most some constant $C(c)$ to an actual group homomorphism; here $C(c)$
depends only on $c$ and not on $n$ or $f$. This gives a positive answer
to a special case of a question posed by Kazhdan and Ziegler.
\end{abstract}

\section{Introduction}

A $c$-quasihomomorphism from a group $G$ to a group $H$ with
a left-invariant metric $d$ is a map $f:G \to H$ such that
$d(f(xy),f(x)f(y)) \leq c$ for all $x,y$ in $G$. A central question
in geometric group theory, first raised by Ulam in
\cite[Chapter 6]{Ulam1960}, is whether there exists an actual homomorphism
$f':G \to H$ such that $d(f(x),f'(x))$ is at most some constant $C$
for all $x$.  Different versions of this question are of interest:
for instance, $C$ may be allowed to depend on $c,G,(H,d)$ but not on
$f$; or $G,(H,d)$ may be restricted to certain classes and 
or $C$ is only allowed to depend on $c$.

A well-known example where the answer to this question is negative is the
case where $G=H=\ZZ$ with the standard metric. Here, quasihomomorphisms
modulo bounded maps are a model of the real numbers \cite{ACampo2021},
and the answer is yes only for those quasihomomorphisms that correspond to 
integers. 

Much literature in this area focusses on {\em quasimorphisms}, which
are quasihomomorphisms into the real numbers $\RR$ with the standard
metric; we refer to \cite{Kotschick2004} for a brief
introduction. In another branch of the research on
quasihomomorphisms $H$ is assumed nonabelian, and 
one of the first positive results on the central question
above is Kazhdan's theorem on $\varepsilon$-representations of amenable
groups \cite{Kazhdan82}. For more recent results on quasihomomorphisms
into nonabelian groups we refer to \cite{FujiwaraKapovich2016} and the
references there.

The following instance of the central question was formulated
by Kazhdan and Ziegler in their work on approximate cohomology
\cite{KazhdanZiegler2018}. 

\begin{question}\label{prob}
Let $c \in \NN$. Does there exist a constant $C = C(c)$ such that the
following holds: For all $n \in \NN$ and all functions $f: \ZZ \longrightarrow \CC^{n\times n}$ such that 
\[
    \forall x,y \in \ZZ: \qquad \mathrm{rk}(f(x+y) - f(x) - f(y))\leq c,
\] 
there exists a matrix $g$ such that 
\[
\forall x \in \ZZ: \qquad \mathrm{rk}(f(x) - x\cdot g)\leq C(c)?
\]
\end{question}

Here, $G$ equals $\ZZ$ and $H$ equals $\CC^{n \times n}$, both with
addition, and the metric on $H$ is defined by $d(A,B):=\rk(A-B)$.
Our main result is an affirmative answer to this question in the special
case where all matrices $f(x)$ are assumed to be {\em diagonal}.

\begin{definition}
Let $(Q,+)$ be an abelian group. For an element $v \in Q^n$, the \emph{Hamming weight} $w_H(v)$ is the number of nonzero entries of $v$.
For a pair of elements $u,v \in Q^n$, their \emph{Hamming distance} is
$w_H(v-u)$. This metric is clearly left-invariant. 
\end{definition}

\begin{definition}\label{def:quasimor}
Let $A$ be another abelian group. A function $f: A \to Q^n$ is called a {\it $c$-quasihomomorphism} if 
\[
    \forall x,y \in A: \qquad w_H(f(x+y)-f(x)-f(y)) \leq c. \qedhere
\]
\end{definition}

\begin{remark}
The map $\operatorname{diag}:\CC^n \to \CC^{n \times n}$ is an isometric embedding from $\CC^n$ with the Hamming metric to $\CC^{n \times n}$ with the rank metric. This connects Definition~\ref{def:quasimor} to Question~\ref{prob}. 
\end{remark}

\begin{definition}\label{def:linapprox}
Let $C\in \NN$ and let $f: A\to Q^n$ be a $c$-quasihomomorphism. A
group homomorphism $h: A\to Q^n$ is a {\it $C$-approximation of $f$} if the Hamming distance between $f$ and $h$ satisfies 
\[ 
    \forall x \in A: \qquad w_H(f(x)-h(x)) \leq C. \qedhere
\]
\end{definition}

We are ready to state our main result.  

\begin{theorem}[Main Theorem]\label{maintheorem}
Let $c\in \NN$. Then there exists a constant $C = C(c) \in \NN$ such
that for all $n \in \NN$ and $c$-quasihomorphisms $f: \ZZ \to \QQ^n$, we have: 
\[
    \forall x \in \ZZ: \qquad w_H(f(x)-x\cdot f(1)) \leq C.
\]
Moreover, we can take $C=28c$. 
\end{theorem}

\begin{remark}
The coefficient $28$ is probably not optimal. %
However, we certainly have that $C(c) \geq c$. Indeed, any map $f:\ZZ \to Q^n$ for which the only nonzero entries of $f(x)$ are among the first $c$, is automatically a $c$-quasimorphism.
\end{remark}

\begin{corollary} \label{cor:main}
    \Cref{maintheorem} also holds with $\QQ$ replaced by any
    torsion-free abelian group $Q$, with the same value of $C=C(c)$.
\end{corollary}
\begin{proof}
    Suppose, for a contradiction, that we have a $c$-quasihomomorphism
    $f:\ZZ \to Q^n$ but $w_H(f(y)-y\cdot f(1)) > C$ for some $y \in
    \ZZ$. Since $Q$ is torsion-free, the natural map $\iota$ from $Q$
    into the $\QQ$-vector space $V:=\QQ \otimes_\ZZ Q$ is injective.
    Consequently, $g:=\iota^n \circ f$ is a $c$-quasihomomorphism
    $\ZZ \to V^n$ with $w_H(g(y)-y\cdot g(1))> C$. Now choose any
    $\QQ$-linear function $\xi:V \to \QQ$ that is nonzero on the
    nonzero entries of $g(y)-y \cdot g(1)$. Then $h:=\xi^n \circ
    g$ is a $c$-quasihomomorphism $\ZZ \to \QQ^n$ with $w_H(h(y)-y
    \cdot h(1))>C$, a contradiction to~\Cref{maintheorem}.
\end{proof}

Theorem \ref{maintheorem} shows that for a $c$-quasihomomorphism $f:
\ZZ \to \QQ^n$, the group homomorphism $\tilde{f}: \ZZ\to \QQ^n$
defined by $\tilde{f}(x) = x\cdot f(1)$ gives a $C$-approximation for
some constant $C\in \NN$ independent on $n$. However, $\tilde{f}$ need
not be the homomorphism closest to $f$, as the next example shows.

\begin{example} 
Let $c=1$ and $n\geq 3$. Define $f: \ZZ \to \QQ^n$ to be 
\[
    f(x) = \left( \left\lfloor\frac{2x+2}{5}\right\rfloor, \left\lfloor\frac{x+2}{5}\right\rfloor, \alpha_x, 0, \ldots, 0\right),
\]
where $\alpha_x$ is arbitrary if $5 \mid x$, and $\alpha_x = 0$ otherwise. This is a $1$-quasihomomorphism.  
	
Note that $w_H(f(x)-x\cdot f(1)) \leq 3$ where equality is sometimes
achieved. However, there also exist $2$-approximations of $f$. For instance, letting $v=(\frac{2}{5},\frac{1}{5},0,\ldots, 0)\in \QQ^n$, one verifies that 
\[
    w_H(f(x)-x\cdot v) \leq 2 \qquad \forall x \in \ZZ. %
\]
In fact, we can show that for every $1$-quasihomomorphism $f: \ZZ \to \QQ^n$ there exists a $2$-approximation, as claimed in \cite{KazhdanZiegler2018}. The proof will appear in future work. 
\end{example}

On the other hand, the following shows that the best possible
approximation of a given qusimorphism $f$ is at most twice as close as the homomorphism $x \mapsto x\cdot f(1)$.

\begin{remark}
Suppose that a map $f:\ZZ \to \QQ^n$ has a $C'$-approximation $h$. Then $h(x)= x\cdot v$ for some $v \in \QQ^n$, and 
\[
    w_H(f(x)-x\cdot v) \leq C' \qquad \forall x \in \NN.
\]
    Substituting $x=1$ yields $w_H(f(1)-v) \leq C'$. Thus
\[
    w_H(f(x)-x\cdot f(1)) \leq w_H(f(x)-x\cdot v) + w_H(x\cdot v -x\cdot
    f(1)) \leq 2C'. \qedhere
\]
\end{remark}

\begin{remark}
A result similar to Theorem \ref{maintheorem} is easily proven in
finite characteristic if we allow the constant $C$ to depend on the
characteristic. Let $K$ be a field of characteristic $p>0$, and let $f:
\ZZ \to K^n$ be a $c$-quasihomomorphism.  Then there exists a constant
$C=C(p,c)$ such that $w_H(f(x)-x \cdot f(1))\leq C$, for all $x\in \ZZ$. 

To see this, we observe that for all $u,v \in \ZZ$ with $u \geq 1$, we have $$
    w_H(f(uv) - uf(v)) \leq (u-1)c.
$$ 
This follows by repeatedly applying the inequality $w_H(f(uv) - f((u-1)v) - f(v)) \leq c$ if $u > 1$; the case $u = 1$ is trivial.

For $x = kp+r$ with $k \in \ZZ$ and $0 \leq r \leq p-1$, we
have 
$$
    w_H(f(x) - x f(1)) = w_H(f(kp+r) - r f(1));
$$ here we
have used that $p \cdot f(1)=0$. We rewrite the latter as
$$
    w_H(f(kp+r) - f(kp) - f(r) + f(kp) + f(r) - r f(1)).
$$
We
have $w_H(f(kp+r) - f(kp) - f(r)) \leq c$; 
$w_H(f(kp)) \leq (p-1)c$ using our observation with $u=p, v
= k$; and also $w_H(f(r) - r f(1)) \leq (p-2)c$ (in the case $r > 0$) using our main observation. In total, this gives $w_H(f(x) - x f(1)) \leq 2(p-1)c$, so we can take $C = 2(p-1)c$.
\end{remark}

The remainder of this paper is organized as follows. In
Section~\ref{sec:Almost} we prove an auxiliary result of independent
interest: maps from a finite abelian group into a torsion-free
group that are almost a homomorphism, are in fact almost zero. Then,
in Section~\ref{sec:Proof}, we apply this auxiliary result to the
component functions of a $c$-quasimorphism $\ZZ \to \QQ^n$ to prove the
Main Theorem.

\section{Almost homomorphisms are almost zero} \label{sec:Almost}

Let $A$ be a finite abelian group and let $H$ be a torsion-free abelian group. The only homomorphism $A \to H$ is the zero map. The following proposition says that maps that are, in a suitable sense, close to being homomorphisms, are in fact also close to the zero map.

\begin{prop}\label{abeliangroup}
Let $a$ be a positive integer, $A$ an abelian group of order
$a$, $H$ a torsion-free abelian group, $q \in [0,1]$, and $f:A \to H$
a map. Suppose that the {\em zero set}
$$
    Z(f):=\{b \in A \mid f(b)=0\} 
$$
has cardinality at most $qa$. Then the {\em problem set}
$$
    P(f):=\{(b,c) \in A\times A \mid f(b+c) \neq f(b)+f(c)\}
$$
has cardinality at least $\left(\frac{a-qa}{2}\right)\cdot \left(\frac{a-qa}{2}+1\right)$. 
\end{prop}

The contraposition of this statement says that if $P(f)$ is a small fraction of $a^2$, so that
$f$ can be thought of as an (additive) ``almost homomorphism'' $A \to H$,
then $q$ must be close to $1$ so that $f$ is essentially zero.

\begin{proof}
Since $H$ is torsion-free, it embeds into the $\QQ$-vector space $V:=\QQ \otimes_\ZZ H$. By basic linear algebra, there exists a $\QQ$-linear function $\xi:V \to \QQ$ such that $\xi(f(b)) \neq 0$ for all $b \not \in Z(f)$, so that $Z(\xi \circ f)=Z(f)$. Since $P(\xi \circ f) \subseteq P(f)$, it suffices to prove the proposition for $\xi \circ f$ instead of $f$. In other words, we may assume from the beginning that $H=\QQ$.

Set 
$$
    B:=\{b \in A \mid f(b)>0\}. 
$$
Let $\lambda_1>\lambda_2>\ldots>\lambda_k>0$ be the distinct
values in $f(B)$, and for each $i=1,\ldots,k$ set 
$$
    B_i:=\{b \in B \mid f(b)=\lambda_i\} \text{ and } n_i:=|B_i|;
$$
as well as $n:=n_1+\cdots+n_k=|B|$. 

Now for each $c \in B_1$ and each $b \in B$ we have
$$
    f(b)+f(c)=f(b)+\lambda_1>\lambda_1 
$$
so that the left-hand side is not in $f(B)$ and in particular not equal to $f(b+c)$. We have thus found $n_1 \cdot (n_1+\cdots+n_k)$ pairs $(b,c) \in P(f)$ with $c \in B_1$. 

Next, suppose $(b,c)$ is a pair with $c \in B_2$, $b \in B$, and $(b,c) \not \in P(f)$.
Then 
$$
    f(b+c)=f(b)+f(c)>f(c)=\lambda_2 
$$
and hence $b+c \in B_1$. But given $c$, there are at most $n_1$ values of $b$ with $b+c \in B_1$. (Note that here we have used that $A$ is a group.) Hence we have at least $n_2 \cdot
(n_2+\cdots+n_k)$ pairs $(b,c) \in P(f)$ with $c \in B_2$. 

Similarly, we find at least $n_i \cdot (n_i+\cdots+n_k)$ pairs $(b,c) \in P(f)$ with $c \in B_i$. In total, we have therefore found at least
\begin{equation}\label{eq:upper_triangular}
     \sum_{i=1}^k n_i \cdot (n_i+\cdots+n_k) \geq \frac{n(n+1)}{2}
\end{equation}
pairs in $P(f)$; see Figure \ref{jans_proof_picture}. 

Let $B':=\{b' \in A \mid f(b')<0\}$ and $n' := |B'|$.  Repeating
the same argument above with $B'$ and $n'$, we find at least
$n'(n'+1)/2$ further pairs in $P(f)$, disjoint from those
found above. Since $|Z(f)|\leq qa$, we have $n+n'\geq a(1-q)$.
Therefore 
\[
    |P(f)|\geq \frac{n(n+1)}{2}+\frac{n'(n'+1)}{2}=\frac{n^2+n'^2}{2}+\frac{n+n'}{2}\geq \left(\frac{n+n'}{2}\right)^2+\frac{n+n'}{2},
\]
where the second inequality is the Cauchy-Schwarz inequality 
\[ 
    \left(n^2+n'^2\right)\left(\frac{1}{2^2}+\frac{1}{2^2}\right)\geq
\left(\frac{n}{2} +\frac{n'}{2}\right)^2. 
\] 
Since $n+n'\geq a(1-q)$, we conclude that 
\begin{equation*}
|P(f)| \geq \left(\frac{a-qa}{2}\right)\cdot \left(\frac{a-qa}{2}+1\right). 
    \qedhere  
\end{equation*} 
\end{proof}

\begin{figure}[b]
	\centering
	\includegraphics[valign=t]{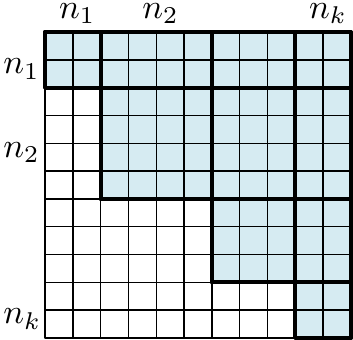}
	\includegraphics[valign=t]{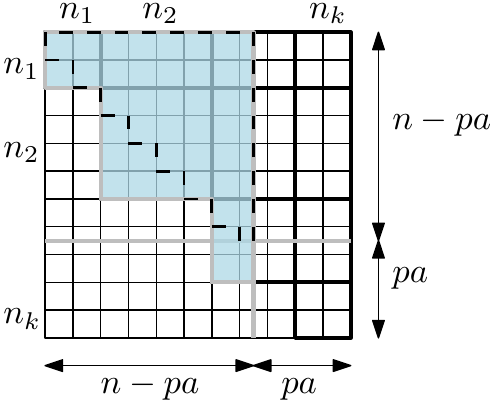}
	\caption{On the left, a graphical proof of the inequality 
	\eqref{eq:upper_triangular}: the left-hand side is the number
	of small squares in the shaded region, the right-hand side is
	the number of squares on or above the main diagonal.
	On the right, a proof of the inequality
	\eqref{eq:ineq2}: the left-hand side is the area of the shaded region, the right-hand side, the area enclosed by the dashed line.}
	
	\label{jans_proof_picture}
\end{figure}

\begin{remark}
The lower bound in Proposition \ref{abeliangroup} is sharp.  Let
$a=2k+1\in \ZZ$, consider $A := \ZZ/a\ZZ$ and define $f:A \to \ZZ$ as
$f(x):=$the representative of $x+a\ZZ$ in $\{-k,\ldots,0,\ldots,k\}$. We take $q=\frac{Z(f)}{a} = \frac{1}{2k+1}$. Then
$f(x+y)=f(x)+f(y)$ if and only if the right-hand side is still inside
the interval $\{-k,\ldots,k\}$, and a straightforward count shows that
this is the case for $3k^2+3k+1$ pairs $(x,y) \in A^2$. Hence $P(f)$
has size $k\cdot (k+1)$, which equals $\left(\frac{a-qa}{2}\right)
\cdot \left(\frac{a-qa}{2}+1\right)$. A similar construction for $a=2k$
yields a problem set of size $\frac{a^2}{4} = k^2$, which equals the ceiling of the lower bound $\frac{a^2}{4}-\frac{1}{4}$.
\end{remark}

Below, we will use the following
strengthening of Proposition \ref{abeliangroup}:

\begin{prop}\label{abeliangroupprime}
Let $a,A,H,q$ and $f$ be as in \Cref{abeliangroup}.
Furthermore, let $p \in [0,\frac{1-q}{2})$ and let $S \subseteq A$ be a subset of cardinality at most $pa$.
Then the set
\[
    P_S(f):=\{(b,c) \in A\times A \mid f(b+c) \neq f(b)+f(c) \text{ and } b+c \notin S\}.
\]
has cardinality at least $\left(\frac{a(1-q-2p)}{2}\right)\cdot \left(\frac{a(1-q-2p)}{2}+1\right)$. 
\end{prop}
\begin{proof} 
Keep the notation from the proof of \Cref{abeliangroup}.
Recall $n=|B|$ and $n' = |B'|$. Note that for a fixed $b$,
there can be at most $pa$ choices of $c$ with $b+c \in S$.
We then find at least $n_i\cdot(n_i+\cdots+n_k-pa)$ pairs
$(b,c) \in P_S(f)$ with $b \in B_i$. Letting $k'\leq k$ be
the largest index for which the second factor
$(n_{k'}+\cdots+n_k-pa)$ is nonnegative, as in the proof of \Cref{abeliangroup}, we find that $B$
contributes at least 
\begin{align}
    \sum_{i=1}^{k'}{n_i\cdot(n_i+\cdots+n_k-pa)} &=
    \sum_{i=1}^{k'}{n_i\cdot(n-n_1-\cdots-n_{i-1}-pa)}
    \notag \\
    &\geq (n-pa)(n-pa+1)/2  \label{eq:ineq2}
\end{align}
to $P_S(f)$; see Figure~\ref{jans_proof_picture}. Similarly, $B'$ contributes at least
$(n'-pa)(n'-pa+1)/2$, and these contributions
are disjoint. The desired inequality follows as in the
proof of \Cref{abeliangroup} but with $n,n'$
replaced by $n-pa,n'-pa$. \end{proof}

The key ingredient for the proof of \Cref{maintheorem} is the following
corollary of \Cref{abeliangroupprime}. Here, and in the rest of the paper,
we write $[a]$ for the set $\{1,2,\ldots,a\}$.

\begin{corollary} \label{cor:nonzeroAndPeriodicGivesProblems}
Let $p, q \in [0,1]$ such that $p < \frac{1-q}{2}$. Let $f:[2a] \to \QQ$ such that:
\begin{enumerate}
	\item $|Z(f)| \leq qa$, where $Z(f) \coloneqq \{x \in [a] \mid f(x) = 0\}$ is the zero set. %
	\item $|NP(f)| \leq pa$, where $NP(f) \coloneqq \{x \in [a] \mid f(x+a) \neq f(x)\}$ is the nonperiodicity set.
\end{enumerate}
Then 
\[
|P(f)| \geq \frac{(1-q-2p)^2}{4}a^2 + \frac{(1-q-2p)}{2}a,
\]
where 
\[
P(f) = \{(x,y) \in [a]\times [a] \mid f(x+y) \neq f(x)+f(y)\}.
\]
\end{corollary}

\begin{proof}
Let $\tilde{f}$ be the restriction of $f$ to the interval
$[a]$, and identify $\ZZ/ {a\ZZ}$ with $[a]$ with the group
operation $*$ defined by $x*y := x+y \ (\operatorname{mod} \ a)$. 

Let $S = NP(f)$, and apply \Cref{abeliangroupprime} to $\tilde{f}$. We find that  
\[
P_S(\tilde{f}) = \{(b,c) \in \ZZ/ {a\ZZ} \times \ZZ/ {a\ZZ} \mid \tilde{f}(b*c) \neq \tilde{f}(b)+\tilde{f}(c) \text{ and } b*c \notin S\}
\]
has cardinality at least $\frac{(1-q-2p)^2}{4}a^2 + \frac{(1-q-2p)}{2}a$. Since $b*c \notin S$ implies that $\tilde{f}(b*c) = f(b+c)$, this set is contained in the problem set $P(f)$.
\end{proof}

\section{Proof of the Main Theorem} \label{sec:Proof}

The main goal of this section is to prove Theorem \ref{maintheorem}. 
We start with some definitions. 

\begin{definition} Let $1<a$, and $f:[2a] \to \QQ$. We define the following \emph{problem sets of $f$}:
$$
    P(f):= \{(x,y) \in [a] \times [a] \mid f(x+y) \neq f(x) + f(y)\},
$$
and
$$
	P_1(f) := \{x \in [a] \mid f(x+1) \neq f(x) + f(1)\},
$$	
and
$$
	P_a(f) := \{x \in [a] \mid f(x+a) \neq f(x) + f(a)\}.
$$
Furthermore, we recall that the {\em zero set} of $f$ is
defined as 
$$ 
	Z(f):=\{x \in [a] \mid f(x) =0\}.
$$
\end{definition}

The following proposition says that $P_1(f),P_a(f),P(f)$
cannot be simultaneously small. 

\begin{prop} \label{conj:variation} Let $p,q \in (0,1)$ such
that $p<\frac{1-q}{2}$, $a\in \NN$ with $1<a$, and let $f:[2a] \to \QQ$ such that $f(a) \neq af(1)$. If 
$$
    |P_1(f)| \leq qa \quad \text{and} \quad  |P_a(f)| \leq pa
$$
then 
$$
	|P(f)|\geq F(p,q) \cdot a^2,
$$
where 
\begin{equation*}
     F(p,q) = \frac{(1-q-2p)^2}{4}. \qedhere
\end{equation*}
\end{prop}

\begin{proof}
Without loss of generality we can assume $f(a)=0$ and hence
$f(1) \neq 0$. Indeed, suppose we have shown the statement
for every $\tilde{f}$ with $\tilde{f}(a)=0$. Then for any
$f:[2a] \to \QQ$ with $f(a) \neq af(1)$, we take
$\tilde{f}:[2a] \to \QQ$ to be $\tilde{f}(x) = af(x) - x f(a)$. Now we observe that $\tilde{f}(a) = 0 \neq a\tilde{f}(1)$, and that $P(f)=P(\tilde{f})$, $P_1(f)=P_1(\tilde{f})$, $P_a(f)=P_a(\tilde{f})$.

We write $Z(f)=\{x_1<\cdots<x_m\}$. Note that for $1\leq i <m$, one of the elements $x_i,x_i+1,\ldots,x_{i+1}-1$ needs to be in $P_1(f)$ since $f(x_{i+1}) \neq f(x_i) + (x_{i+1}-x_i)f(1)$. Likewise, at least one of the elements $1, 2, \ldots, x_1-1$ needs to be in $P_1(f)$. Thus we have 
$$
    |Z(f)| \leq |P_1(f)| \leq qa,
$$
and by assumption we have $|NP(f)| = |P_a(f)| \leq pa$. Now we can apply \Cref{cor:nonzeroAndPeriodicGivesProblems} to conclude.
\end{proof}

We now prove \Cref{maintheorem}.

\begin{proof}[Proof of the Main Theorem.]
Consider a $c$-quasihomomorphism $f=(f_1,\ldots,f_n):\ZZ \to \QQ^n$. Our goal is to show that for every $a \in \ZZ$ we have $w_H(f(a)-a\cdot f(1)) \leq C$ for some constant $C$ depending only on $c$. We start with the case $a >0$.

Write $M:=w_H(f(a)-af(1))$. Without loss
of generality, we have $f_1(a)\neq af(1)$, \ldots,
$f_M(a)\neq af(1)$. We will show that $M \leq C'$ for some
constant $C'$ depending on $c$ only. 
To this end, fix small parameters $p,q \in (0,1)$ (to be
optimized over later) and consider the 
restrictions $f_i:[2a] \to \QQ$ of the components of $f$.
By \Cref{conj:variation}, for every $i$, we have that either
\begin{enumerate}[label=(\roman*)]
	\item\label{it:P}  $|P_1(f_i)| > qa$, or
	\item\label{it:P1} $|P_a(f_i)| > pa$, or
	\item\label{it:Pa} $|P(f_i)| \geq F(p,q)a^2$.
\end{enumerate}
Let $m_0$ be the number of coordinates $i$ such that \ref{it:Pa} holds. We define $m_1$ and $m_2$ analogously, for \ref{it:P} and \ref{it:P1} respectively.
	
By counting the number of triples $(i,x,y) \in [M] \times [a] \times [a]$ such that $f_i(x+y)-f_i(x)-f_i(y) \neq 0$ in two ways, we see that
$$
	\sum_{x=1}^{a}\sum_{y=1}^{a}{w_H\left(f(x+y)-f(x)-f(y)\right)} = \sum_{i=1}^{M}{|P(f_i)|}.
$$
Because $f$ is a $c$-quasihomomorphism, the left hand side
is at most $a^2c$. On the other hand, the right hand side is
at least $m_0F(p,q)a^2$, so 
$$
    a^2c \geq
    \sum_{x=1}^{a}\sum_{y=1}^{a}{w_H\left(f(x+y)-f(x)-f(y)\right)}
    = \sum_{i=1}^{M}{|P(f_i)|}  \geq  m_0F(p,q)a^2.
$$
So we obtain $m_0 \leq \frac{c}{F(p,q)}$. Similarly we find 
$$
	ac \geq
	\sum_{x=1}^{a}{w_H\left(f(x+1)-f(x)-f(1)\right)} =
	\sum_{i=1}^{M}{|P_1(f_i)|}  > m_1qa, 
$$
so that $m_1 < \frac{c}{q}$. Finally, 
$$
	ac \geq \sum_{x=1}^{a}{w_H\left(f(x+a)-f(x)-f(a)\right)} = \sum_{i=1}^{M}{|P_a(f_i)|} > m_2pa.
$$	
So $m_2 < \frac{c}{p}$. But now $M = m_0 + m_1 + m_2 < c
(\frac{1}{F(p,q)}+\frac{1}{q}+\frac{1}{p})=:C'$. 

The case $a=0$ is easy: we have 
\[ w_H(f(0))=w_H(f(0)-f(0)-f(0)) \leq c.\] 

Finally, let us consider the case $a<0$. Then
\begin{align*}
    w_H(f(a)-a\cdot f(1)) \leq& w_H(f(a)+f(-a)-f(0))+w_H(f(0)) \\
    &+w_H(f(-a)-(-a)\cdot f(1)) \\
    \leq& 2c+C'=:C.
\end{align*}

This completes the proof.
\end{proof}

To get the explicit bound $28c$ from Theorem~\ref{maintheorem}, we 
minimize the function 
\[
2+\frac{1}{q}+\frac{1}{p}+\frac{1}{F(p,q)}=
2+\frac{1}{q}+\frac{1}{p}+\frac{4}{(1-q-2p)^2}.
\]
Note that this function is strictly convex for $(p,q) \in \RR_{>0}^2$,
so that it has at most one minimum in the positive orthant. We find
this by setting the partial derivatives equal to zero and solving for
$p,q$. The minimum is $\approx 27.6817$ and attained at $(p,q) \approx
(0.1167, 0.16500)$.

\bibliographystyle{alpha}
\bibliography{quasimorphisms}{}

\begin{thebibliography}{A'C21}

\bibitem[A'C21]{ACampo2021}
N.~A'Campo.
\newblock A natural construction for the real numbers.
\newblock {\em Elem. Math.}, 76(3):89--105, 2021.

\bibitem[FK16]{FujiwaraKapovich2016}
K.~Fujiwara and M.~Kapovich.
\newblock On quasihomomorphisms with noncommutative targets.
\newblock {\em Geom. Funct. Anal.}, 26(2):478--519, 2016.

\bibitem[Kaz82]{Kazhdan82}
D.~Kazhdan.
\newblock On {$\varepsilon $}-representations.
\newblock {\em Israel J. Math.}, 43(4):315--323, 1982.

\bibitem[Kot04]{Kotschick2004}
D.~Kotschick.
\newblock What is{$\dots$}a quasi-morphism?
\newblock {\em Notices Amer. Math. Soc.}, 51(2):208--209, 2004.

\bibitem[KZ18]{KazhdanZiegler2018}
D.~Kazhdan and T.~Ziegler.
\newblock Approximate cohomology.
\newblock {\em Selecta Math. (N.S.)}, 24(1):499--509, 2018.

\bibitem[Ula60]{Ulam1960}
S.~M. Ulam.
\newblock {\em A collection of mathematical problems}.
\newblock Interscience Tracts in Pure and Applied Mathematics, no. 8.
  Interscience Publishers, New York-London, 1960.

\end{thebibliography}

\end{document}